\documentclass[final,1p,times]{elsarticle}

\usepackage{microtype}
\usepackage{amsfonts,amssymb,amsmath,amsthm}
\usepackage{graphicx}
\usepackage[colorlinks=true]{hyperref}
\usepackage[english]{babel}
\usepackage{csquotes}
\usepackage{pgfplots}
\pgfplotsset{compat=1.18}
\usepackage{nicefrac}

\theoremstyle{plain}
\newtheorem{theorem}{Theorem}

\theoremstyle{definition}
\newtheorem{definition}{Definition}
\newtheorem{remark}{Remark}

\DeclareMathOperator*{\argmin}{arg\,min}

\journal{Applied Mathematics Letters}

\begin{document}

\begin{frontmatter}
    \title{An Equivalence result for sketched Anderson Acceleration and sketched GMRES\tnoteref{t1}}

    \author{Alberto Bucci} %

    \affiliation{organization={School of Mathematics, University of Edinburgh},%
                addressline={Peter Guthrie Tait Rd}, 
                city={Edinburgh},
                postcode={EH9 3FD}, 
                state={Scotland},
                country={UK}}

\author{Fabio Durastante} %

\affiliation{organization={Dipartimento di Matematica, Università di Pisa},%
            addressline={Largo Bruno Pontecorvo, 5}, 
            city={Pisa},
            postcode={56127}, 
            state={PI},
            country={Italy}}
    \tnotetext[t1]{{AB is supported by the UK’s Engineering and Physical Sciences Research Council (EPSRC grant EP/Z533786/1). AB and FD acknowledge the INdAM GNCS project with CUP E53C25002010001. FD acknowledges the MUR Excellence Department Project awarded to the Department of Mathematics, University of Pisa, CUP I57G22000700001.}}

    \begin{abstract}
	In this paper we present an equivalence result between a randomized version of Anderson Acceleration and 
	of randomized GMRES for linear problems. Namely, we extend the classical result of~\citeauthor{MR2831068} (\citeyear{MR2831068})
	to the case in which the least-squares problem in Anderson Acceleration is solved in a sketched space 
	defined by a random projection. This equivalence suggests potential avenues for further research in the 
	design and analysis of randomized acceleration methods.
    \end{abstract}

    \begin{keyword}
    Sketching \sep GMRES \sep Anderson Acceleration 
    \MSC[2020] 65F10 \sep 65B99 \sep 60B20
    \end{keyword}
    
\end{frontmatter}

\section{Introduction}\label{sec:intro}

Anderson acceleration (AA), originally introduced in~\cite{MR184447}, is a widely used extrapolation strategy 
for fixed-point iterations and nonlinear equations; see, for instance,~\cite{MR4497809}. For linear problems, i.e.,
when applied to a linear fixed-point iteration for solving the linear system
\begin{equation}\label{eq:linear_system}
A \mathbf{x}=\mathbf{b}, \quad A\in\mathbb{R}^{n\times n}, \quad \mathbf{b}\in\mathbb{R}^n,
\end{equation}
AA is known to be closely 
related to Krylov subspace methods. Several convergence and 
equivalence properties are now well understood~\cite{MR3979953,MR3841161}. On the Krylov side, the Generalized Minimal Residual (GMRES) 
method is the reference method for nonsymmetric systems. Following a recent trend of interest in randomized numerical 
linear algebra~\cite{MR4189294}, both randomized and sketched variants of GMRES~\cite{MR4434356,MR4760964,MR4882482} and randomized alternating 
variants of AA~\cite{barnafi2026twolevelsketchingalternatinganderson} have been introduced. This motivates the present investigation, which extends the classical results of~\cite{MR2831068} to establish equivalence between appropriately defined variants of sketched AA and sketched GMRES{. A key point in the argument is that the result is independent of how the basis is constructed or orthogonalized. It therefore applies to different sketched GMRES variants, regardless of their orthogonalization strategy.}%

\section{Sketched GMRES and Sketched Anderson Acceleration}\label{sec:sketched_gmres}

We start by recalling a randomized construction of GMRES based on sketched Gram--Schmidt orthogonalization and randomized Arnoldi iteration, as described in~\cite{MR4434356}.
For this construction, we need the notion of an $\varepsilon$-embedding for a subspace~\cite[\S 8.1]{MR4189294}.
\begin{definition}\label{def:epsilon_embedding}
A matrix $\Theta\in\mathbb{R}^{s\times n}$ is an $\varepsilon$-embedding for a subspace $\mathcal{V}\subset\mathbb{R}^n$ if
\[
(1-\varepsilon)\|\mathbf{v}\|_2^2 \leq \|\Theta \mathbf{v}\|_2^2 \leq (1+\varepsilon)\|\mathbf{v}\|_2^2, \quad \forall \mathbf{v}\in\mathcal{V}.
\]
We define $\| \cdot \|_{\Theta} \triangleq  \|\Theta \cdot\|_2$ and $\langle \cdot, \cdot \rangle_{\Theta} \triangleq  \langle \Theta \cdot, \Theta \cdot \rangle$ as the sketched norm and inner product.%
\end{definition}

Choose an embedding dimension $s\ll n$ and a sketching matrix $\Theta\in\mathbb{R}^{s\times n}$, so that inner products in the Krylov space are approximated by
the sketched inner product $\langle \mathbf{u},\mathbf{v}\rangle\approx\langle  \mathbf{u}, \mathbf{v}\rangle_\Theta$.
The method can be used to construct a basis that is orthonormal in the sketched inner product,
or {exploit} the solution of the sketched least-squares problem to reduce 
the depth of the orthogonalization process, see, e.g.~\cite{MR4760964}. 
In the following, we recall the construction with full orthogonalization, 
and then comment on the equivalence result with a truncated orthogonalization process.
For a target iteration count $m$, and an initial guess $\mathbf{x}_0$, define $\mathbf{q}_1=(\mathbf{b} - A \mathbf{x}_0)/\|\mathbf{b} - A \mathbf{x}_0\|_\Theta $ and, for $i=1,\dots,m-1$, compute
$\mathbf{w}_{i+1}=A\mathbf{q}_i$, and then apply one sketched Gram--Schmidt step:
{\begin{equation*}
\begin{split}
\mathbf{p}_{i+1}=\Theta \mathbf{w}_{i+1},\quad S_i=\Theta Q_i, \quad
[H]_{(1:i,i)}=\argmin_{\mathbf{y}\in\mathbb{R}^i}\|S_i \mathbf{y}-\mathbf{p}_{i+1}\|_2,
\quad\text{i.e., } h_{ji}=\langle \mathbf{q}_j,A\mathbf{q}_i\rangle_\Theta,\ j\leq i,\\
\mathbf{q}'_{i+1}=\mathbf{w}_{i+1}-Q_i [H]_{(1:i,i)},\qquad
h_{i+1,i}=\|\mathbf{q}'_{i+1}\|_\Theta,\qquad \mathbf{q}_{i+1}=\mathbf{q}'_{i+1}/h_{i+1,i}.
\end{split}
\end{equation*}}
Collecting {the} coefficients, {\(h_{ji}\), \(j=1,\ldots,i\), which are the coefficients obtained by projecting \(\mathbf{w}_{i+1}\) onto the current basis in the sketched least-squares sense,} gives an upper Hessenberg matrix $\overline{H}_m\in\mathbb{R}^{m\times (m-1)}$ and a basis matrix $Q_m=[\mathbf{q}_1,\dots,\mathbf{q}_m]$ satisfying the Arnoldi relation $A Q_{m-1}=Q_m \overline{H}_m$.

The randomized GMRES iterate is then obtained by solving {$\displaystyle \argmin_{\mathbf y\in\mathcal{K}(A,\mathbf{r}_0)}
\|\mathbf r_0-A\mathbf y\|_\Theta$ on the Krylov subspace $\mathcal{K}(A,\mathbf{r}_0) = \operatorname{span}\{\mathbf{r}_0,A \mathbf{r}_0,\ldots,A^{k-1}\mathbf{r}_0\}$, through} the small least-squares problem
$\displaystyle \mathbf{y}_{m-1}=\argmin_{\mathbf{z}\in\mathbb{R}^{m-1}}\|\overline{H}_m \mathbf{z}-h_{1,0}\mathbf{e}_1\|_2$,
and setting $\mathbf{x}_{m-1}=Q_{m-1}\mathbf{y}_{m-1}$,
where $h_{1,0}=\| \mathbf{b} - A \mathbf{x}_0\|_\Theta$ and $\mathbf{e}_1$ is the first canonical vector. %
In exact arithmetic, if $\Theta$ is an $\varepsilon$-embedding for the generated subspace, this construction is equivalent to minimizing the sketched residual norm $\min_{\mathbf{z}\in\mathbb{R}^{m-1}}\|AQ_{m-1}\mathbf{z}-\mathbf{b}\|_\Theta$
and therefore it provides a quasi-optimal residual minimizer in the original norm up to the embedding distortion factor.

AA is a nonlinear extrapolation technique that can be applied to fixed-point iterations. 
Given a fixed-point iteration of the form $\mathbf{x}_{k+1} = G(\mathbf{x}_k)$, $k = 0, 1, 2, \ldots$, and $\mathbf{x}_0 \in \mathbb{R}^n$,
AA constructs an extrapolated iterate by combining the previous iterates and their images under $G$.
At iteration $k$, AA considers the last $m+1$ iterates $\mathbf{x}_{k-m}, \ldots, \mathbf{x}_k$ and their corresponding residuals $\mathbf{r}_i = G(\mathbf{x}_i) - \mathbf{x}_i$ for $i = k-m, \ldots, k$.
It then solves a least-squares problem to find coefficients $\alpha_i$ that minimize the norm of the extrapolated residual
\begin{equation}\label{eq:aa_least_squares}
\textstyle \min_{\boldsymbol{\alpha}} \left\| \sum_{i=k-m}^k \alpha_i \mathbf{r}_i \right\|_2 = \min_{\boldsymbol{\alpha}} \left\|  R_k \boldsymbol{\alpha} \right\|_2, \quad \text{ subject to } \sum_{i=k-m}^k \alpha_i = 1.
\end{equation}
The next iterate is then given by
$\mathbf{x}_{k+1} = \sum_{i=k-m}^k \alpha_i G(\mathbf{x}_i)$.
An equivalent unconstrained formulation is obtained by eliminating one coefficient with
$\alpha_k = 1-\sum_{j=k-m}^{k-1}\alpha_j$. Defining $\Delta R_k \triangleq  [\mathbf{r}_{k-m}-\mathbf{r}_k,\ldots,\mathbf{r}_{k-1}-\mathbf{r}_k]\in\mathbb{R}^{n\times m}$ and $\boldsymbol{\beta}\triangleq [\alpha_{k-m},\ldots,\alpha_{k-1}]^\top\in\mathbb{R}^{m}$, problem~\eqref{eq:aa_least_squares} is equivalent to
\begin{equation}\label{eq:aa_unconstrained_least_squares}
\textstyle \min_{\boldsymbol{\beta}\in\mathbb{R}^{m}}\|\mathbf{r}_k+\Delta R_k\boldsymbol{\beta}\|_2.
\end{equation}
When Anderson Acceleration is applied to the linear fixed-point iteration 
$\mathbf{x}_{k+1} = M\mathbf{x}_k + \mathbf{b}$, it is equivalent to GMRES 
applied to $(I-M)\mathbf{x} = \mathbf{b}$, where $I$ denotes the identity 
matrix of proper size, provided that: no truncation is 
performed, both methods start from the same initial guess $\mathbf{x}_0$, 
and the Krylov subspace dimension matches the Anderson Acceleration depth.
In terms of convergence histories, the equivalence result given by~{\cite[Theorem 2.2]{MR2831068}} says that, as 
long as the GMRES residual does not stagnate and is strictly decreasing at each step, 
untruncated AA is effectively following the same Krylov information as GMRES. More precisely, 
at iteration $k$, the AA affine combination of past iterates reproduces the $k$th GMRES iterate. 

\subsection{A randomized version of Anderson Acceleration}

Let us now consider a randomized version of Anderson Acceleration, where the least-squares problem~\eqref{eq:aa_least_squares} is solved 
in a sketched inner product defined again by a sketching matrix $\Theta\in\mathbb{R}^{s\times n}$ with $s\ll n$, namely
\begin{equation}\label{eq:sketched_aa_least_squares}
\textstyle \min_{\boldsymbol{\alpha}} \left\| \sum_{i=k-m}^k \alpha_i \mathbf{r}_i \right\|_\Theta = \min_{\boldsymbol{\alpha}} \left\|  \Theta R_k \boldsymbol{\alpha} \right\|_2, \quad \text{ subject to } \sum_{i=k-m}^k \alpha_i = 1.
\end{equation}
As in~\eqref{eq:aa_unconstrained_least_squares}, eliminating $\alpha_k = 1-\sum_{j=k-m}^{k-1}\alpha_j$ yields the equivalent unconstrained sketched problem.

This can be seen as a sketched version of Anderson Acceleration (sAA), where the extrapolation coefficients are computed by minimizing the norm of the extrapolated residual in the sketched space.
The main advantage of sketching is computational: by reducing the dimension of the least-squares problem, each AA step becomes cheaper while approximately preserving the geometry of the residual subspace. If $\Theta$ is an accurate embedding for $\operatorname{span}\{\mathbf r_{k-m},\dots,\mathbf r_k\}$, the coefficients computed from~\eqref{eq:sketched_aa_least_squares} remain close to those of standard AA~\cite[Theorem 4.1]{barnafi2026twolevelsketchingalternatinganderson}, so that similar convergence behavior may be expected at a lower per-iteration cost, much as in randomized GMRES~\cite{MR4434356}. The tradeoff is that the true residual norm is no longer minimized directly. As a result, convergence in the original norm need not be monotone, and excessively small sketches may produce poor extrapolation coefficients, slowing convergence or reducing robustness. In practice, larger sketches improve fidelity, while smaller sketches reduce computational cost.
\section{Equivalence result}\label{sec:equivalence}
The extension of {\cite[Theorem 2.2]{MR2831068}} to the sketched setting relies on the $\varepsilon$-embedding assumption of Definition~\ref{def:epsilon_embedding}, which ensures norm equivalence on the residual Krylov space. This allows the original argument to carry over and implies that sGMRES and sAA minimize the same residual functional on the same Krylov subspaces.

\begin{theorem}\label{thm:equivalence_sketched}
Let $
G(\mathbf x)=M\mathbf x+\mathbf b$, $M\in\mathbb R^{n\times n}$, with $A = I-M$ nonsingular. Let $\Theta\in\mathbb R^{s\times n}$ be a sketching matrix 
realizing a $\varepsilon$-embedding (Definition~\ref{def:epsilon_embedding}) for
$\mathcal{V} = \mathcal K_{k}(A,\mathbf r_0) = \operatorname{span}\{\mathbf r_0,\ldots,A^{k-1}\mathbf r_0\}$, where $\mathbf r_0\triangleq \mathbf b-A\mathbf x_0$. 
Let $\mathbf x_j^{\mathrm{sGMRES}}$ denote the iterates of sketched GMRES applied to~\eqref{eq:linear_system} that is, $\mathbf x_j^{\mathrm{sGMRES}}=\mathbf x_0+\mathbf z_j$,
for $\mathbf z_j
=\argmin_{\mathbf z\in\mathcal K_j(A,\mathbf r_0)}
\|\mathbf r_0-A\mathbf z\|_\Theta
$. Let $\mathbf x_j^{\mathrm{sAA}}$ denote the iterates of untruncated sketched Anderson acceleration, where the coefficients at step $j$ are obtained from~\eqref{eq:sketched_aa_least_squares}. Assume that for some $k>0$,
\begin{equation}\label{eq:residual_decrease}
\mathbf r_{k-1}^{\mathrm{sGMRES}}\neq 0,
\text{ and }
\|\mathbf r_j^{\mathrm{sGMRES}}\|_\Theta
<
\|\mathbf r_{j-1}^{\mathrm{sGMRES}}\|_\Theta,
\qquad
j=1,\ldots,k,
\end{equation}
where $\mathbf r_j^{\mathrm{sGMRES}}=\mathbf b-A\mathbf x_j^{\mathrm{sGMRES}}$.
Then $\sum_{i=0}^{k}
\alpha_i^{(k)}
\mathbf x_i^{\mathrm{sAA}}
=\mathbf x_k^{\mathrm{sGMRES}}$, and $\mathbf x_{k+1}^{\mathrm{sAA}}
=G\!\left(\mathbf x_k^{\mathrm{sGMRES}}\right)$.
\end{theorem}

\begin{proof}
Let $\mathbf r(\mathbf x)\triangleq \mathbf b-A\mathbf x$. Since $A=I-M$, we have
$\mathbf r(\mathbf x)
=\mathbf b-(I-M)\mathbf x
=M\mathbf x+\mathbf b-\mathbf x
=G(\mathbf x)-\mathbf x$. Thus the residuals used in sAA coincide exactly with the residuals 
of the linear system. For every iterate $\mathbf x_i^{\mathrm{sAA}}$, 
define
\begin{equation}\label{eq:residual_representation}
\mathbf r_i{^{\mathrm{sAA}}}\triangleq G(\mathbf x_i^{\mathrm{sAA}})-\mathbf x_i^{s\mathrm{AA}}
=\mathbf r(\mathbf x_i^{s\mathrm{AA}})
=\mathbf r_0-A(\mathbf x_i^{s\mathrm{AA}}-\mathbf x_0),
\end{equation}
{resulting {in} $\mathbf{r}_0 = \mathbf r_0^{\mathrm{sGMRES}} = \mathbf{r}_0^{\mathrm{sAA}}$}.
At iteration $k$, {sketched} Anderson acceleration computes coefficients $\boldsymbol{\alpha}^{(k)}$ such that
\begin{equation}\label{eq:least_square_sketched}
\boldsymbol{\alpha}^{(k)} = (\alpha_0^{(k)},\ldots,\alpha_k^{(k)})^\top = \argmin_{\boldsymbol{\alpha}\in\mathbb{R}^{k+1} \,:\, \sum_{i=0}^{k}\alpha_i=1}\, {\textstyle \left\|
\sum_{i=0}^{k}
\alpha_i
\mathbf r_i{^{\mathrm{sAA}}}
\right\|_\Theta.}
\end{equation}
Using the representation~\eqref{eq:residual_representation}, and the fact that the coefficients satisfy the affine constraint, we write
\[
\textstyle {\textbf{g}_k \triangleq }\sum_{i=0}^{k}\alpha_i^{(k)}\mathbf r_i{^{\mathrm{sAA}}}
=\sum_{i=0}^{k}\alpha_i^{(k)}\left(\mathbf r_0-A(\mathbf x_i^{\mathrm{sAA}}-\mathbf x_0)\right)
=\mathbf r_0
-A\left(
\sum_{i=0}^{k}\alpha_i^{(k)}(\mathbf x_i^{\mathrm{sAA}}-\mathbf x_0)
\right).
\]
If we denote by $\mathbf y\triangleq \sum_{i=1}^{k}\alpha_i^{(k)}(\mathbf x_i^{\mathrm{sAA}}-\mathbf x_0)$,
then the sketched Anderson least-squares problem~\eqref{eq:least_square_sketched} would be equivalent to the sGMRES minimization $\min_{\mathbf y\in\mathcal S_k}
\|\mathbf r_0-A\mathbf y\|_\Theta$, whenever $\mathcal S_k\triangleq \operatorname{span}\{\mathbf x_1^{\mathrm{sAA}}-\mathbf x_0,\dots,\mathbf x_k^{\mathrm{sAA}}-\mathbf x_0\}$ is equal to $\mathcal{K}_{k}(A,\mathbf{r}_0)$, in other words
\[
\textstyle \mathbf{x}_{j+1}^{\mathrm{sAA}} = \sum_{i=0}^j \alpha_i^{(j)} G(\mathbf{x}_i^{\mathrm{sAA}}) = G\left(\sum_{i=0}^j \alpha_i^{(j)} \mathbf{x}_i^{\mathrm{sAA}} \right) = G\left( \mathbf{x}_j^{\mathrm{sGMRES}}\right).
\]
To conclude, we now need to prove the identity $\mathcal S_k=\mathcal K_k(A,\mathbf r_0)$.
 
We prove by induction that $\mathcal S_j=\mathcal K_j(A,\mathbf r_0)$, $j=1,\dots,k
$. For $j=1$, sAA coincides with the underlying fixed-point iteration. Since
$\mathbf x_1^{\mathrm{sAA}}=G(\mathbf x_0)=\mathbf x_0+\mathbf r_0$, we obtain
$\mathbf x_1^{\mathrm{sAA}}-\mathbf x_0=\mathbf r_0$.
Therefore
$\mathcal S_1=\operatorname{span}\{\mathbf r_0\}=\mathcal K_1(A,\mathbf r_0)$.
Assume now that $\mathcal S_j=\mathcal K_j(A,\mathbf r_0)$ for some $j<k$.
Let $\mathbf y_j\triangleq \sum_{i=1}^{j}\alpha_i^{(j)}(\mathbf x_i^{\mathrm{sAA}}-\mathbf x_0)$ 
be the Anderson minimizer~\eqref{eq:least_square_sketched} at step $j$.

\[
\begin{split}
\mathbf{x}_{j+1}^{\mathrm{sAA}} - \mathbf{x}_0 & %
= M \mathbf{x}_j^{\mathrm{sGMRES}} + \mathbf{b} - \mathbf{x}_0 = \mathbf{b} - (I-M)\mathbf{x}_j^{\mathrm{sGMRES}} + \mathbf{x}_j^{\mathrm{sGMRES}} - \mathbf{x}_0 
 = \mathbf{r}_j^{\mathrm{sGMRES}} + \mathbf{y}_j,
\end{split}
\]
By the induction hypothesis, $\mathbf y_j\in \mathcal K_j(A,\mathbf r_0)
\subseteq
\mathcal K_{j+1}(A,\mathbf r_0)$, while the previous argument shows that
$\mathbf r_j^{\mathrm{{sGMRES}}}
\in
\mathcal K_{j+1}(A,\mathbf r_0)$. Hence $\mathbf x_{j+1}^{\mathrm{{s}AA}}-\mathbf x_0
=
\mathbf y_j+\mathbf r_j^{\mathrm{{sGMRES}}}
\in
\mathcal K_{j+1}(A,\mathbf r_0)$.
Consequently, $\mathcal S_{j+1}\subseteq\mathcal K_{j+1}(A,\mathbf r_0)$. 
{To show that $\mathcal S_{j+1}$ expands strictly in dimension over $\mathcal S_j$, suppose by contradiction that $\mathbf x_{j+1}^{\mathrm{sAA}}-\mathbf x_0\in\mathcal S_j$. 
By the induction hypothesis, $\mathcal S_j = \mathcal K_j(A, \mathbf r_0)$. Since $\mathbf y_j \in \mathcal K_j(A, \mathbf r_0)$ and we assumed $\mathbf x_{j+1}^{\mathrm{sAA}}-\mathbf x_0 \in \mathcal K_j(A, \mathbf r_0)$, the relation $\mathbf{x}_{j+1}^{\mathrm{sAA}} - \mathbf{x}_0 = \mathbf{r}_j^{\mathrm{sGMRES}} + \mathbf{y}_j$ dictates that $\mathbf r_j^{\mathrm{sGMRES}} \in \mathcal K_j(A, \mathbf r_0)$. By definition, the sGMRES residual is $\mathbf r_j^{\mathrm{sGMRES}} = \mathbf r_0 - A\mathbf z_j$ for some $\mathbf z_j \in \mathcal K_j(A, \mathbf r_0)$. Because $\mathbf r_j^{\mathrm{sGMRES}}$ and $\mathbf r_0$ both belong to $\mathcal K_j(A, \mathbf r_0)$, it follows that $A\mathbf z_j \in \mathcal K_j(A, \mathbf r_0)$. This leaves only two possibilities:
either $\mathbf z_j \in \mathcal K_{j-1}(A, \mathbf r_0)$, meaning the minimizer at step $j$ was already available at step $j-1$, which yields $\|\mathbf r_j^{\mathrm{sGMRES}}\|_\Theta = \|\mathbf r_{j-1}^{\mathrm{sGMRES}}\|_\Theta$; 
or $\mathbf z_j \notin \mathcal K_{j-1}(A, \mathbf r_0)$, meaning its unique expansion in the basis $\{\mathbf r_0, A\mathbf r_0, \dots, A^{j-1}\mathbf r_0\}$ of $\mathcal K_j(A, \mathbf r_0)$ has a non-zero coefficient on $A^{j-1}\mathbf r_0$. In this case, $A\mathbf z_j \in \mathcal K_j(A, \mathbf r_0)$ forces $A^j\mathbf r_0 \in \mathcal K_j(A, \mathbf r_0)$, indicating that the Krylov subspace is invariant (i.e., $\mathcal K_{j+1}(A, \mathbf r_0) = \mathcal K_j(A, \mathbf r_0)$) and yielding $\|\mathbf r_{j+1}^{\mathrm{sGMRES}}\|_\Theta = \|\mathbf r_j^{\mathrm{sGMRES}}\|_\Theta$.
Both scenarios directly contradict the strict residual decrease assumption. We conclude that $\mathbf x_{j+1}^{\mathrm{sAA}}-\mathbf x_0 \notin \mathcal S_j$. Hence, $\dim(\mathcal S_{j+1}) = \dim(\mathcal S_j)+1 = j+1$. Since $\mathcal S_{j+1}\subseteq\mathcal K_{j+1}(A,\mathbf r_0)$ and both spaces have dimension $j+1$, it follows that $\mathcal S_{j+1} = \mathcal K_{j+1}(A,\mathbf r_0)$. This completes the induction and proves that $\mathcal S_k=\mathcal K_k(A,\mathbf r_0)$.}
\end{proof}
{The embedding assumption in Theorem~\ref{thm:equivalence_sketched} applies only to the Krylov subspace $\mathcal{K}_k(A,\mathbf r_0)$, rather than to the full ambient space $\mathbb{R}^n$. Thus, the sketch dimension $s$ depends on $k$, the distortion $\varepsilon$, and the failure probability $\delta$, allowing $s\ll n$ when $k\ll n$. For $\varepsilon<1$, the embedding also guarantees nondegeneracy on the Krylov space, which is the property needed for the exact equivalence.}

\begin{remark}
The equivalence depends only on the Krylov subspace $\mathcal{K}_k(A,\mathbf{r}_0)$ and the nondegeneracy of the sketched norm, not the chosen basis. Thus, any variant minimizing the same sketched residual over this subspace, such as sketched GMRES with short-orthogonalization, preserves the equivalence.
\end{remark}

{Truncating AA to a memory depth $m$ breaks its exact equivalence with GMRES~\cite{desterck2023andersonaccelerationkrylovmethod}. Nevertheless, AA$(m)$ retains a residual-minimization property that allows its residuals to be compared with those of full, unrestarted GMRES, whose minimization is performed over the entire Krylov space $\mathcal K_k(A,\mathbf r_0)$. This comparison should not, however, be interpreted as an equivalence with restarted GMRES$(m)$: AA$(m)$ uses a sliding window of previous iterates, whereas restarted GMRES$(m)$ discards the current Krylov space and starts a new one at each restart. The resulting approximation spaces therefore have a more intricate relationship~\cite{desterck2023andersonaccelerationkrylovmethod}. Establishing an analogous exact or approximate equivalence between truncated sAA and restarted sGMRES would require a separate analysis and is beyond the scope of the present work.}

\section{Numerical experiments}\label{sec:numerical_experiments}

We numerically validate here the classical result in 
{\cite[Theorem 2.2]{MR2831068}} and the new one in Theorem~\ref{thm:equivalence_sketched} by monitoring the identity $\mathbf{x}_{k+1}^{\mathrm{AA}}=G\left(\mathbf{x}_k^{\mathrm{GMRES}}\right)$ and $\mathbf{x}_{k+1}^{\mathrm{sAA}}=G\left(\mathbf{x}_k^{\mathrm{sGMRES}}\right)$, respectively, at each iteration
on the test problem \texttt{airfoil1.mtx} ($4253\times 4253$), 
with random right-hand side, initial guess $\mathbf{x}_0=\mathbf{0}$, 
and $50$ iterations. Both theorem{s} predict, for each admissible index $k$, the iterate identity
$\mathbf{x}_{k+1}^{\mathrm{AA}}=G\left(\mathbf{x}_k^{\mathrm{GMRES}}\right)$ and
$\mathbf{x}_{k+1}^{\mathrm{sAA}}=G\left(\mathbf{x}_k^{\mathrm{sGMRES}}\right)$.
To assess this statement directly, we monitor the relative discrepancies for $k = 0,\ldots,49$
\(
\eta_k =  \nicefrac{\left\|\mathbf{x}_{k+1}^{\mathrm{AA}}-G\left(\mathbf{x}_k^{\mathrm{GMRES}}\right)\right\|_2}{\max\!\left(\left\|G\left(\mathbf{x}_k^{\mathrm{GMRES}}\right)\right\|_2,\,\varepsilon_{\mathrm{mach}}\right)},\) and
\( \eta_k^{\mathrm{s}} =  \nicefrac{\left\|\mathbf{x}_{k+1}^{\mathrm{sAA}}-G\left(\mathbf{x}_k^{\mathrm{sGMRES}}\right)\right\|_2}{\max\!\left(\left\|G\left(\mathbf{x}_k^{\mathrm{sGMRES}}\right)\right\|_2,\,\varepsilon_{\mathrm{mach}}\right)},\)
where $\varepsilon_{\mathrm{mach}}$ is the machine precision of the floating-point 
arithmetic used in the computations. The denominator is included to avoid division 
by zero and to provide a relative measure of the discrepancy. We have implemented
everything in \textsc{Julia}, and we have used the \texttt{Float64} type for double precision 
(53-bit) and the \texttt{BigFloat} type for quadruple precision (113-bit). 
The sketching matrix $\Theta$ is a Gaussian embedding with $s=100$ rows. 
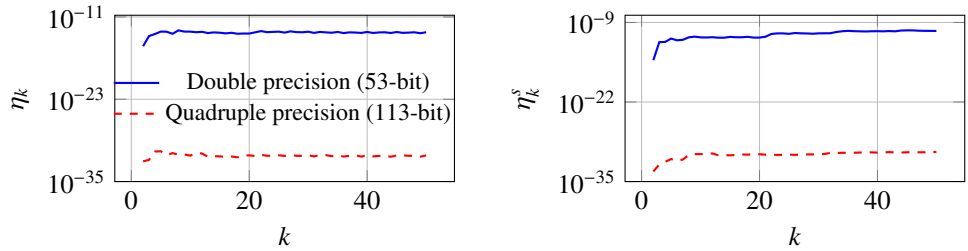
\begin{figure}[htbp]
\centering
\begin{tikzpicture}
\begin{semilogyaxis}[
	width=0.45\textwidth,
	height=0.28\textwidth,
	at ={(0.0,0)},
	xlabel={$k$},
	ylabel={$\eta_k$},
	grid=major,
	ymin=1e-35,
	legend style={draw=none, fill=none, at={(0.5,0.28)}, align=left, anchor=south, font=\small},
]
\addplot+[thick, mark=none] table[x=iter, y=eq_relerr, col sep=comma] 
{residuals_standard_airfoil1_randomrhs_maxit50_tol1.0e-08_seed1_m50.csv};
\addlegendentry{Double precision (53-bit)}

\addplot+[thick, dashed, mark=none] table[x=iter, y=eq_relerr, col sep=comma] {residuals_standard_airfoil1_randomrhs_maxit50_tol1.0e-08_seed1_m50_bf113.csv};
\addlegendentry{Quadruple precision (113-bit)}

\end{semilogyaxis}

\begin{semilogyaxis}[
	width=0.45\textwidth,
	height=0.28\textwidth,
	at ={(0.5\textwidth,0)},
	xlabel={$k$},
	ylabel={$\eta_k^s$},
	legend style={at={(0.5,0.1)},anchor=south},
	grid=major,
	ymin=1e-35,
]
\addplot+[thick, mark=none] table[x=iter, y=eq_relerr, col sep=comma] 
{residuals_airfoil1_randomrhs_maxit50_tol1.0e-08_k100_seed1_m50.csv};

\addplot+[thick, dashed, mark=none] table[x=iter, y=eq_relerr, col sep=comma] %
{residuals_airfoil1_randomrhs_maxit50_tol1.0e-08_k100_seed1_m50_bf113.csv};
\end{semilogyaxis}

\end{tikzpicture}
\caption{Validation of Theorem~\ref{thm:equivalence_sketched}: relative discrepancies $\eta_k$ for $\mathbf{x}_{k+1}^{\mathrm{AA}}=G(\mathbf{x}_k^{\mathrm{GMRES}})$ and $\eta_k^s$ for $\mathbf{x}_{k+1}^{\mathrm{sAA}}=G(\mathbf{x}_k^{\mathrm{sGMRES}})$ over all tested iterations on \texttt{airfoil1.mtx}.}
\label{fig:conv_standard_airfoil1}
\end{figure}
The results are shown in Figure~\ref{fig:conv_standard_airfoil1}, where we 
can see that the identity is satisfied up to the used precision with a deterioration given by the error accumulation along the iterates.

\section{Conclusion}\label{sec:conclusions}

We have presented a sketched variant of Anderson acceleration and established its equivalence with sketched GMRES for linear fixed-point iterations. The result extends the classical Anderson acceleration--GMRES correspondence to a randomized setting and shows that, under standard non-stagnation assumptions, the Anderson extrapolated iterate at step $k$ coincides with the $k$th sketched GMRES iterate. Numerical experiments support the theoretical analysis and illustrate the agreement between the two approaches. Possible directions for future work include the study of limited-memory variants and a more detailed investigation of other connected sequence acceleration strategies~\cite{MR3841161}, such as Shanks sequences~\cite{MR3841161} and the $\varepsilon$-vector algorithm~\cite{MR4530335}.
\smallskip

{\textbf{Acknowledgments.} We thank the reviewers for their comments, which helped us enhance the overall presentation of our work.}
\bibliographystyle{elsarticle-num-names} 
\bibliography{equivaa}

\end{document}